\documentclass{amsart}

\usepackage{amsmath,amssymb,amsthm}

\usepackage{pinlabel}

\newtheorem{prop}{Proposition}[section]
\newtheorem{thm}[prop]{Theorem}
\newtheorem{lem}[prop]{Lemma}

\theoremstyle{definition}

\newtheorem{rem}[prop]{Remark}

\newtheorem*{ack}{Acknowledgements}

\def\co{\colon\thinspace}

\newcommand{\bfa}{\mathbf{a}}
\newcommand{\obfa}{\overline{\mathbf{a}}}

\newcommand{\bfb}{\mathbf{b}}

\newcommand{\bfc}{\mathbf{c}}

\newcommand{\C}{\mathbb{C}}
\newcommand{\CP}{\mathbb{C}\mathrm{P}}

\newcommand{\rmd}{\mathrm{d}}

\newcommand{\rme}{\mathrm{e}}

\newcommand{\HH}{\mathbb{H}}

\newcommand{\rmi}{\mathrm{i}}
\newcommand{\bfi}{\mathbf{i}}

\newcommand{\bfj}{\mathbf{j}}

\newcommand{\bfk}{\mathbf{k}}

\newcommand{\bfp}{\mathbf{p}}

\newcommand{\R}{\mathbb{R}}
\newcommand{\RP}{\mathbb{R}\mathrm{P}}

\newcommand{\SO}{\mathrm{SO}}

\newcommand{\bft}{\mathbf{t}}

\newcommand{\bfu}{\mathbf{u}}
\newcommand{\obfu}{\overline{\mathbf{u}}}

\newcommand{\bfv}{\mathbf{v}}

\newcommand{\bfx}{\mathbf{x}}
\newcommand{\tX}{\tilde{X}}
\newcommand{\Xh}{X_{\mathrm{h}}}

\newcommand{\bfy}{\mathbf{y}}

\newcommand{\Z}{\mathbb{Z}}

\DeclareMathOperator{\Real}{Re}

\begin{document}

\author[P.~Albers]{Peter Albers}
\address{Mathematisches Institut, Universit\"at Heidelberg,
Im Neuenheimer Feld 205, 69120 Heidelberg}
\email{palbers@mathi.uni-heidelberg.de}

\author[H.~Geiges]{Hansj\"org Geiges}
\address{Mathematisches Institut, Universit\"at zu K\"oln,
Weyertal 86--90, 50931 K\"oln, Germany}
\email{geiges@math.uni-koeln.de}

\author[K.~Zehmisch]{Kai Zehmisch}
\address{Mathematisches Institut, WWU M\"unster, Einsteinstra{\ss}e 62,
48149 M\"unster, Germany}
\email{kai.zehmisch@uni-muenster.de}

\thanks{This research is part of a project
in the SFB/TRR 191 `Symplectic Structures in Geometry, Algebra and Dynamics',
funded by the DFG}

\title[Reeb dynamics inspired by Katok's example]{Reeb dynamics inspired
by Katok's example in Finsler geometry}

\date{}

\begin{abstract}
Inspired by Katok's examples of Finsler metrics with a small
number of closed geodesics, we present two results
on Reeb flows with finitely many periodic orbits.
The first result is concerned with a contact-geometric 
description of magnetic flows on the $2$-sphere found recently by
Benedetti. We give a simple interpretation of that
work in terms of a quaternionic symmetry. In the second part,
we use Hamiltonian circle actions on symplectic manifolds
to produce compact, connected contact manifolds in dimension at least five
with arbitrarily large numbers of periodic Reeb orbits.
This contrasts sharply
with recent work by Cristofaro-Gardiner, Hutchings and
Pomerleano on Reeb flows in dimension three.
With the help of Hamiltonian plugs and
a surgery construction due to Laudenbach we reprove a result of
Cieliebak: one can produce Hamiltonian flows in dimension at least five
with any number of periodic orbits; in dimension three, with any
number greater than one.
\end{abstract}


\maketitle


\section{Introduction}
In a much-cited paper, Katok~\cite{kato73} constructed
non-reversible Finsler metrics on spheres and projective spaces
with small numbers of closed geodesics (and an ergodic geodesic flow)
by deforming a Riemannian metric all
of whose geodesics are closed. The geometry of these examples
has been analysed in great detail by Ziller~\cite{zill82}.
For the Katok metrics on the $2$-sphere, a contact-geometric interpretation
of the geodesic flow on the unit tangent bundle (or its double
cover, the $3$-sphere) as a Reeb flow was given by Harris and
Paternain~\cite{hapa08}.

Geodesic (or rather cogeodesic) flows can be interpreted as Hamiltonian
flows with respect to the canonical symplectic form on the cotangent bundle
and the Hamiltonian function given by the square of the fibre norm.
The Katok deformation of the Riemannian metric into
a Finsler metric translates into a deformation of the Hamiltonian
function.

Magnetic flows are generalisations of these geodesic flows;
here the canonical symplectic form is modified by adding a closed
$2$-form lifted from the base of the cotangent bundle.
Benedetti~\cite{bene16} recently extended the work of
Harris--Paternain and gave an interpretation of magnetic flows
on the $2$-sphere (including their deformations into flows with
only two periodic orbits) in terms of Reeb flows.

Katok takes a Riemannian manifold whose geodesic flow is periodic
and alters the flow with the help of an $S^1$-action on that manifold.
For instance, one starts with the geodesic flow of the round metric
on the $2$-sphere $S^2$, and modifies it with the help of
the rotation of $S^2$ about a fixed axis into a non-reversible
Finsler metric of Randers type. Only the equator (traversed in
either direction) survives as a closed geodesic.

In the present paper we discuss two constructions inspired by these
examples. In Section~\ref{section:quaternionic} we show that
Benedettis's work admits a strikingly simple interpretation in
terms of a quaternionic symmetry of the $3$-sphere. It is well
known that the standard Reeb flow on the $3$-sphere, which defines
the Hopf fibration, can be perturbed into a Reeb flow with only
two periodic orbits, seen for instance as the Reeb flow
on an irrational ellipsoid in $\R^4$. We show that Benedetti's
magnetic flows are nothing but quaternionic rotates
of this example. The detailed formulation of the result is
given in Theorem~\ref{thm:quaternionic}. Here is a pr\'ecis.

Magnetic flows are Hamiltonian flows in the cotangent bundle of a
manifold with respect to a symplectic form obtained by adding a closed
$2$-form lifted from the base to the canonical symplectic form.
As Hamiltonian function we take the fibre norm squared, and we
study the dynamics on the unit cotangent bundle. When the base manifold
is the $2$-sphere~$S^2$, it is convenient to pass to the double cover $S^3$
of the unit cotangent bundle.

Now, on $S^3$ there is a natural $S^2$-family of contact forms induced from
a $2$-sphere of complex structures on $\R^4$, regarded as the
space of quaternions. Writing $\bfi,\bfj,\bfk$ for the standard
quaternionic units and $r$ for the radial
coordinate on~$\R^4$, we set $\alpha_{\bfi}=-2\,\rmd r\circ\bfi|_{TS^3}$;
the contact forms $\alpha_{\bfj}$ and $\alpha_{\bfk}$ are defined
analogously. The triple $(\alpha_{\bfi},\alpha_{\bfj},\alpha_{\bfk})$
is a taut contact sphere in the sense of~\cite{gego95}.

We show that magnetic flows on (the unit cotangent bundle of) $S^2$, for
the magnetic term being a multiple of the standard area form,
are simply the Reeb flows in the circular family
$\alpha^{\theta}:=\cos\theta\,\alpha_{\bfi}+\sin\theta\,\alpha_{\bfj}$
of contact forms. All these Reeb flows are periodic.
The Reeb flow of $\alpha_{\bfi}$ corresponds
to an infinite magnetic term and defines the Hopf
fibration, as the double cover of the unit cotangent bundle of~$S^2$.
The Reeb flow of $\alpha_{\bfj}$ corresponds to a zero magnetic term
and constitutes the geodesic flow. The deformation of the Hopf fibration
into a flow with only two periodic orbits is then
effected simultaneously on the whole circular family of Reeb flows.

The second part of the paper is concerned with the
construction of Reeb and Hamiltonian flows with large but finite
numbers of periodic orbits.
In Section~\ref{section:finite} we build on the idea of
modifying a given Reeb flow on a bundle with the help
of an $S^1$-action on the base manifold. Specifically, we
start with the Reeb flow along the fibres of an $S^1$-bundle
over a symplectic manifold coming from the classical
Boothby--Wang construction~\cite{bowa58}, cf.~\cite[Section~7.2]{geig08}.
We then modify this periodic flow with the help of a
Hamiltonian $S^1$-action on the base.
In dimension five, for instance, this allows one to realise
any integer $\geq 3$ as the number of periodic Reeb orbits on
some closed, connected contact manifold.
This contrasts sharply with recent results of Cristofaro-Gardiner,
Hutchings and Pomerleano~\cite{chp17}. In essence they
established that, in dimension three, a contact manifold
carries either two or infinitely many periodic Reeb orbits.

Using in addition a surgery construction for Hamiltonian flows
due to Laudenbach~\cite{laud97} and Hamiltonian plugs~\cite{grz16}, we can
give a new proof of a result by Cieliebak~\cite{ciel97}:
Any non-negative integer (in dimension three:
any integer $\geq 2$) can be realised as the number of
closed characteristics on some hypersurface in the standard
symplectic $\R^{2n+2}$, i.e.\ as periodic orbits of a Hamiltonian flow
(Theorem~\ref{thm:finite}).

Sections~\ref{section:quaternionic} and~\ref{section:finite}
are to a large part independent of each other.
\section{A quaternionic view of magnetic flows on the $2$-sphere}
\label{section:quaternionic}
On the cotangent bundle $T^*S^2$ of the $2$-sphere $S^2$ we consider
the canonical Liouville form $\lambda$, defined by
$\lambda_{\bfv^*}=\bfv^*\circ T\pi$ at $\bfv^*\in T^*S^2$,
where $\pi\co T^*S^2\rightarrow S^2$ denotes the bundle projection.
Given any (closed) $2$-form $\sigma$ on $S^2$,
which is referred to as the \emph{magnetic field},
we can define a symplectic form
\[ \omega_{\sigma}:=\rmd\lambda-\pi^*\sigma\]
on $T^*S^2$. The Hamiltonian vector field $X_H$ corresponding to
a smooth function $H\co T^*S^2\rightarrow\R$ is then defined by
\[ \omega_{\sigma}(X_H,\,.\,)=-\rmd H.\]
If $g$ is a Riemannian metric on $S^2$ and $\|\,.\,\|_*$ the dual norm on
$T^*S^2$, the flow of the Hamiltonian vector field defined by the
Hamiltonian function $H(\bfv^*)=\frac{1}{2}\|\bfv^*\|_*^2$ is
called the \emph{magnetic flow} of the pair $(g,\sigma)$.

\begin{rem}
\label{rem:geodesic}
The Hamiltonian flow for the symplectic form $\rmd\lambda$
is precisely the (co-)geodesic flow of~$g$. In
\cite[Theorem~1.5.2]{geig08} this is verified for the unit cotangent bundle
$H^{-1}(1/2)$ of any Riemannian manifold. Since $H$ scales quadratically
in the fibres of the cotangent bundle, the Hamiltonian vector field scales
linearly in the fibres, as it should.
\end{rem}

Now let $\sigma_0$ be the standard symplectic form on $S^2$
of total area~$4\pi$, corresponding to the round metric $g_0$
on $S^2$ of constant Gau{\ss} curvature~$1$. We want to study
the magnetic flow of the pairs $(g_0,s\sigma_0)$ for $s\in[0,\infty)$.
For $s=0$ we have the geodesic flow; the projected flow lines on $S^2$ are the
great circles. As the strength of the magnetic field
increases, the projected flow lines on $S^2$ acquire an
increasing left drift, causing the great circles for $s=0$ to
become circles on $S^2$ of smaller and smaller radius. For an
`infinite' magnetic field, these circles become points, and the
magnetic flow is along the fibres of $T^*S^2$.

In order to include this case, we consider the circular family
of closed $2$-forms
\[ \sin\theta\,\rmd\lambda+\cos\theta\, \pi^*\sigma_0,\;\;
\theta\in\R/2\pi\Z.\]
For $\theta\equiv 0$ mod~$\pi$, this form is not symplectic.
There is, however, a simple way to extend $\pi^*\sigma_0$
to a symplectic form on $T^*S^2\setminus S^2$
(the complement of the zero section in $T^*S^2$) whose restriction
to the tangent bundle of $ST^*S^2$ coincides with that of $\pi^*\sigma_0$.
This is done as follows.

On $T^*S^2$ we have the canonical Liouville
vector field $Y=Y_{T^*S^2}$ defined by $\rmd\lambda(Y,\,.\,)=\lambda$.
In local coordinates $q_1,q_2$ on $S^2$ and dual coordinates $p_1,p_2$
on $T^*S^2$ we have $\lambda=p_1\,\rmd q_1+p_2\,\rmd q_2$, and hence
$Y=p_1\partial_{p_1}+p_2\partial_{p_2}$ is the fibrewise
radial vector field. We now identify
$\R^+\times ST^*S^2$, where we write $\rho$ for the $\R^+$-coordinate,
with $T^*S^2\setminus S^2$
by sending $\{1\}\times ST^*S^2$ identically to $ST^*S^2$,
and the flow lines of $\rho\partial_{\rho}$ to the flow lines of~$Y$.
Under this identification, the symplectic form $\rmd\lambda$
on $T^*S^2$ pulls back to the symplectisation $\rmd(\rho\lambda_1)$
of the contact form $\lambda_1$ on $ST^*S^2$ given by the
restriction of~$\lambda$.

On the unit cotangent bundle $ST^*S^2$ we have a further contact form,
namely, the connection $1$-form $\alpha$, which satisfies
$\rmd\alpha=\pi^*\sigma_0$;
notice that the Euler class of the bundle is given by
$-[\sigma_0/2\pi]=-2$. We then obtain a circular family of
symplectic forms on $ST^*S^2\setminus S^2\equiv\R^+\times ST^*S^2$ by setting
\[ \omega^{\theta}:=\sin\theta\,\rmd\lambda+\cos\theta\,\rmd(\rho\alpha),
\;\; \theta\in\R/2\pi\Z.\] 

\begin{rem}
\label{rem:Ham-Reeb}
The Hamiltonian $H(\bfv^*)=\frac{1}{2}\|\bfv^*\|_*^2$ on $T^*S^2$
corresponds to the Hamiltonian $\rho^2/2$ on $\R^+\times ST^*S^2$,
for both functions take the value $1/2$ on $ST^*S^2$ and are
homogeneous of degree~$2$ with respect to the flow of $Y\equiv
\rho\partial_{\rho}$. It follows that the Hamiltonian
flow of $\omega^{\theta}$ on $H^{-1}(1/2)=\{\rho=1\}$
coincides with the Reeb flow of the contact form $\sin\theta\,
\lambda_1+\cos\theta\,\alpha$.
\end{rem}

The Hamiltonian function $H$ is fixed throughout our discussion,
so from now on we shall usually speak of `the Hamiltonian flow
of $\omega^{\theta}$' whenever the Hamiltonian $H$ is to be understood.

For $\theta\not\equiv 0$ mod~$\pi$, the Hamiltonian flow of
$\omega^{\theta}$ on the unit cotangent bundle $ST^*S^2=H^{-1}(1/2)$
equals the magnetic flow of $(g_0,-\cot\theta\,\sigma_0)$,
up to constant time reparametrisation; in particular,
for $\theta\equiv\pi/2$ mod~$\pi$, the Hamiltonian flow
is the geodesic flow of~$g_0$. For $\theta\equiv 0$ mod~$\pi$,
the flow rotates the fibres of $ST^*S^2$.

There is a natural double covering of $ST^*S^2$
by the $3$-sphere $S^3$ (Section~\ref{subsection:double}),
and a left-action of $S^3$ on itself coming from the
multiplication of unit quaternions. The $S^1$-action mentioned
in the following main theorem of this section comes from
choosing a circle of unit quaternions (see Section~\ref{subsection:q-action}).
The contact forms $\alpha_{\bfi},\alpha_{\bfj}$ in the theorem are
the ones defined in the introduction.

\begin{thm}
\label{thm:quaternionic}
There are contact forms $\alpha_{\bfi},\alpha_{\bfj}$
on $S^3$ and an $S^1$-action on
$S^3$ sending $\alpha_{\bfi}$ to contact forms
$\alpha^{\theta}=\cos\theta\,\alpha_{\bfi}
+\sin\theta\,\alpha_{\bfj}$, $\theta\in\R/2\pi\Z$,
such that the Reeb flow of $\alpha^{\theta}$, which has all orbits
periodic of the same period~$4\pi$, doubly covers the Hamiltonian flow
of $\omega^{\theta}$ on $ST^*S^2$. Moreover, there is a deformation
$\alpha_{\bfi,\varepsilon}$, $\varepsilon\in[0,1)$, through contact
forms, mapped to a family of contact forms $\alpha^{\theta}_{\varepsilon}$
under the $S^1$-action, with the following properties:
\begin{itemize}
\item[(i)] The Reeb flow of $\alpha^{\theta}_{\varepsilon}$ doubly
covers the Hamiltonian flow of $\omega^{\theta}$
on a level set of a suitably perturbed Hamiltonian
function~$H^{\theta}_{\varepsilon}$.
\item[(ii)] For any irrational value of the parameter $\varepsilon$,
this Hamiltonian flow has precisely two periodic orbits.
\item[(iii)] For $\theta\equiv\pi/2$ $\mbox{\rm mod}$~$\pi$
and $\varepsilon$ sufficiently small,
the Hamiltonian flow of $H^{\theta}_{\varepsilon}$ is the geodesic flow of a
Finsler metric on~$S^2$.
\end{itemize}
\end{thm}

\begin{rem}
There is in fact a whole $S^2$-family of such contact forms,
with the corresponding deformations. For symmetry reasons,
all the dynamical phenomena are present in the smaller $S^1$-family.
\end{rem}

Theorem~\ref{thm:quaternionic} is a variant of~\cite[Theorem~1.3]{bene16}.
The proof there hinges on a subtle construction of
a family of symplectomorphisms
\[ \bigl(\{\|\bfv^*\|_*>s\},\rmd\lambda\bigr)\longrightarrow
\bigl(\{\|\bfv^*\|_*>0\},\omega_{s\sigma_0}\bigr)\]
for $s\in[0,\infty)$. In our proof, by contrast, the variation
in the $s$-parameter (or the $\theta$-parameter in our set-up)
becomes a perfectly straightforward application of a
quaternionic symmetry, and the case $s=\infty$ is included
naturally in the family.

Concerning the perturbation into Hamiltonian flows with only two
periodic orbits, the essential feature of our argument is the
following. Rather than perturbing the Hamiltonian function and
the corresponding level sets, as the formulation of the
theorem suggests, we actually deform the contact forms $\Z_2$-equivariantly
on~$S^3$, so that we always talk about Reeb flows on the fixed
manifold $ST^*S^2$. Only then do we interpret these Reeb flows
as Hamiltonian flows for a deformed Hamiltonian function.
This approach via Reeb flows has the benefit that the quaternionic
symmetry is preserved throughout the deformation, so the fact that the
contact forms in the family $\{\alpha_{\varepsilon}^{\theta}\}_{\theta
\in\R/2\pi\Z}$ share the same Reeb dynamics for any fixed value
of the deformation parameter $\varepsilon$ becomes a tautology.
\subsection{The structure forms on a surface}
Let $\Sigma$ be an oriented surface with a Riemannian metric $g$
and associated complex structure~$J$. Write $\pi\co ST\Sigma
\rightarrow\Sigma$ for the unit tangent bundle of $(\Sigma,g)$.
The Liouville--Cartan forms of $(\Sigma,g)$ are the $1$-forms
$\lambda_1,\lambda_2$ on $ST\Sigma$ defined by 
the following equations, where $\bfv\in ST\Sigma$ and $\bft\in
T_{\bfv}(ST\Sigma)$:
\begin{eqnarray*}
\lambda_1(\bft) & = & g(\bfv,T_{\bfv}\pi(\bft)),\\
\lambda_2(\bft) & = & g(J\bfv,T_{\bfv}\pi(\bft))
                      \;\;=\;\;-g(\bfv,JT_{\bfv}\pi(\bft)).
\end{eqnarray*}
It is well known and easy to check, see~\cite[Section~7.2]{sith67},
that there is a unique $1$-form $\alpha_*$ on $ST\Sigma$ satisfying
the structure equations
\begin{eqnarray*}
\rmd\lambda_1 & = & -\lambda_2\wedge\alpha_*,\\
\rmd\lambda_2 & = & -\alpha_*\wedge\lambda_1.
\end{eqnarray*}
This $1$-form $\alpha_*$ is the connection form on $ST\Sigma$, and
the third structure equation is
\[ \rmd\alpha_*=-(\pi^*K)\,\lambda_1\wedge\lambda_2,\]
where $K$ is the Gau{\ss} curvature of $(\Sigma,g)$.

\begin{rem}
Pick any local orthonormal frame $(\mu_1,\mu_2)$ of $1$-forms on $(\Sigma,g)$.
In a local trivialisation of the unit tangent bundle we may think
of $\lambda_1,\lambda_2$ as
\[ \lambda_1=\cos\theta\,\mu_1+\sin\theta\,\mu_2,\;\;\;\;
\lambda_2=-\sin\theta\,\mu_1+\cos\theta\,\mu_2.\]
The formula above then translates into a quick recipe for computing the
Gau{\ss} curvature $K$ of $(\Sigma,g)$, see~\cite[Section~4.14]{gonz05}.
Define local functions $a_1,a_2$ on $\Sigma$ by
\[ \rmd\mu_1=a_1\, \mu_1\wedge\mu_2,\;\;\;\;
\rmd\mu_2=a_2\,\mu_1\wedge\mu_2.\]
Set $\mu=-a_1\mu_1-a_2\mu_2$. Then $K$ is found via the equation
$\rmd\mu=K\mu_1\wedge\mu_2$.
\end{rem}

We also write $\lambda_1,\lambda_2$ for the
corresponding  $1$-forms on the unit cotangent bundle $ST^*\Sigma$
under the bundle isomorphism $ST\Sigma\rightarrow ST^*\Sigma$
induced by the metric~$g$. The natural orientation on the fibres
of $ST^*\Sigma$ is the one for which this bundle isomorphism
is fibre orientation \emph{reversing}, since dual bundles
have opposite Euler classes. So the connection $1$-form $\alpha^*$
on $ST^*\Sigma$ corresponds to~$-\alpha_*$. This means that
the structure forms $(\lambda_1,\lambda_2,\alpha^*)$ on
$ST^*\Sigma$ satisfy structure equations as on $ST\Sigma$,
but with the minus signs in the three equations removed.

From now on we shall only be working on cotangent bundles,
and we write the connection $1$-form $\alpha^*$ simply as~$\alpha$.

\begin{rem}
The $1$-form $\lambda_1$ on $ST^*\Sigma$ is the restriction
of the canonical Liouville $1$-form $\lambda$ on $T^*\Sigma$.
If $q_1+q_2\bfi$ is a local holomorphic coordinate
on $\Sigma$, and $p_1,p_2$ are the dual coordinates of $q_1,q_2$ on
$T^*\Sigma$, then $\lambda_1=p_1\,\rmd q_1+p_2\,\rmd q_2$ and
$\lambda_2=p_1\,\rmd q_2-p_2\,\rmd q_1$.
\end{rem}
\subsection{The structure forms on $S^2$}
\label{subsection:structureS2}
Let $g_0$ be the round metric on $S^2$ of constant Gau{\ss}
curvature~$1$, that is, the Riemannian metric induced on the unit sphere
$S^2\subset\R^3$ by the standard scalar product $\langle\,.\,,\,.\,\rangle$
on~$\R^3$. We may then think of the unit cotangent
bundle $ST^*S^2$ as the submanifold
\[ ST^*S^2=\bigl\{(\bfx,\bfy)\in\R^6\co|\bfx|=1,\,|\bfy|=1,\,
\langle\bfx,\bfy\rangle=0\bigr\} \]
of $\R^6$, with $(\bfx,\bfy)$
interpreted as the covector $\dot{\bfx}\mapsto\langle\bfy,\dot{\bfx}\rangle$
for $\dot{\bfx}\in T_{\bfx}S^2$.
Notice that with respect to the natural orientation
of~$S^2$, the positive orientation of $ST^*_{\bfx}S^2$
is defined by ordered pairs of covectors of the form
$(\bfx,\bfy)$ and $(\bfx,-\bfx\times\bfy)$.

The pair $(\lambda_1,\lambda_2)$ of Liouville--Cartan forms on $ST^*S^2$ 
in this model is given by
\begin{eqnarray*}
\bigl(\lambda_1\bigr)_{(\bfx,\bfy)}(\dot{\bfx},\dot{\bfy})
 & = & \langle\bfy,\dot{\bfx}\rangle,\\
\bigl(\lambda_2\bigr)_{(\bfx,\bfy)}(\dot{\bfx},\dot{\bfy})
 & = & \langle\bfx\times \bfy,\dot{\bfx}\rangle.
\end{eqnarray*}
With $\alpha$ denoting the connection $1$-form on $ST^*S^2$
we have the structure equations
\begin{eqnarray*}
\rmd\lambda_1 & = & \lambda_2\wedge\alpha,\\
\rmd\lambda_2 & = & \alpha\wedge\lambda_1,\\
\rmd\alpha    & = & \lambda_1\wedge\lambda_2.
\end{eqnarray*}
The last equation shows $\lambda_1\wedge\lambda_2$ to be invariant
under the flow along the fibres of $\pi\co ST^*S^2\rightarrow S^2$,
and from the definition of the $\lambda_i$ we then see
that $\lambda_1\wedge\lambda_2$ is the lift of the area form $\sigma_0$
on $S^2$ corresponding to the metric~$g_0$, that is,
$\rmd\alpha=\pi^*\sigma_0$.
\subsection{Contact forms on $S^3$ induced by the quaternions}
We regard $S^3$ as the unit sphere in the space $\HH\cong\R^4$ of quaternions.
We write $\bfi,\bfj,\bfk$ for the standard quaternionic units,
and we use the same notation for the complex bundle structures
they induce on the tangent bundle of~$\HH$. The
units $\bfi,\bfj,\bfk$ define an identification of 
$\R^3$ with the space of pure imaginary quaternions.
For other aspects of quaternionic notation see
Appendix~\ref{appendix:quaternionic}.

Any element $\bfc=c_1\bfi+c_2\bfj+c_3\bfk\in S^2\subset\R^3$ defines a
complex structure on~$\HH$, that is, $\bfc^2=-\mathrm{id}_{T\HH}$.
This gives rise to a $2$-sphere of contact forms
\[ \alpha_{\bfc}:=-2\,\rmd r\circ\bfc|_{TS^3},\;\; \bfc\in S^2,\]
on $S^3$, where $r$ is the radial function on~$\HH$. The Reeb vector field
of $\alpha_{\bfc}$ is
\[ R_{\bfc}=\frac{1}{2}\bfc\partial_r.\]
For a general discussion of such $2$-spheres of contact forms
see~\cite{gego95,gego09}.
\subsection{The Hopf fibration on $S^3$}
\label{subsection:Hopf}
The $2$-dimensional complex vector space $\C^2$ can be
identified with the space $\HH$ of quaternions via
\[ \C^2\ni(z_0,z_1)\longmapsto z_0+z_1\bfj\in\HH.\]
We consider the Hopf fibration corresponding to this choice
of coordinates, that is,
\[ \C^2\supset S^3\ni(z_0,z_1)\longmapsto [z_0:z_1]\in\CP^1=S^2.\]
The fibre of this Hopf fibration over the point $[z_0:z_1]$
is parametrised by $\rme^{\bfi t/2}(z_0,z_1)$, $t\in\R/4\pi\Z$.
This choice of parametrisation corresponds to regarding
the fibres as the orbits of the Reeb vector field~$R_{\bfi}$.
Notice that with $z_0=x_0+y_0\bfi$, $z_1=x_1+y_1\bfi$ we have
\[ \alpha_{\bfi}=2(x_0\,\rmd y_0-y_0\,\rmd x_0+x_1\,\rmd y_1-y_1\,\rmd x_1)\]
and
\[ R_{\bfi}=\frac{1}{2}(x_0\partial_{y_0}-y_0\partial_{x_0}
+x_1\partial_{y_1}-y_1\partial_{x_1}).\]
From the corresponding expressions for $\alpha_{\bfj}$ and
$\alpha_{\bfk}$ it is easy to check that
\begin{eqnarray*}
\rmd\alpha_{\bfi} & = & \alpha_{\bfj}\wedge\alpha_{\bfk},\\
\rmd\alpha_{\bfj} & = & \alpha_{\bfk}\wedge\alpha_{\bfi},\\
\rmd\alpha_{\bfk} & = & \alpha_{\bfi}\wedge\alpha_{\bfk}.
\end{eqnarray*}
In other words, the triple $(\alpha_{\bfi},\alpha_{\bfj},\alpha_{\bfk})$
of $1$-forms on $S^3$ satisfies the same structure equations
as the triple $(\tilde{\alpha},\tilde{\lambda}_1,\tilde{\lambda}_2)$
obtained by lifting the structure forms for the constant curvature
$1$ metric on $S^2$ from $ST^*S^2$ to~$S^3$.

Here is a quick sketch, cf.~\cite[Section~3]{gego02},
how one may now argue that there
is a diffeomorphism $S^3\rightarrow S^3$ that pulls back
$(\tilde{\alpha},\tilde{\lambda}_1,\tilde{\lambda}_2)$ to
$(\alpha_{\bfi},\alpha_{\bfj},\alpha_{\bfk})$; this diffeomorphism will
be constructed explicitly in the next section.

Regard the $1$-forms $\alpha_{\bfi}-\tilde{\alpha},
\alpha_{\bfj}-\tilde{\lambda}_1,\alpha_{\bfk}-\tilde{\lambda}_2$
in the obvious way
as $1$-forms on $S^3\times S^3$. The structure equations imply
that this triple of $1$-forms generates a differential ideal,
and hence defines a $3$-dimensional foliation on $S^3\times S^3$.
Since either triple $(\tilde{\alpha},\tilde{\lambda}_1,\tilde{\lambda}_2)$
and $(\alpha_{\bfi},\alpha_{\bfj},\alpha_{\bfk})$ defines a coframe on $S^3$,
the leaves of this
foliation are graph-like with respect to either $S^3$-factor.
Because of the compactness of $S^3\times S^3$, the inclusion of a leaf in
$S^3\times S^3$, followed by the projection onto the first factor,
is a covering. Since $S^3$ is simply connected, this covering
must be trivial, so the leaves are indeed graphs of the desired
diffeomorphism.

Since the $1$-forms $\alpha_{\bfc}$ are invariant under the
\emph{right}-action by $S^3$ on itself, see Remark~\ref{rem:right-invariant},
this diffeomorphism is unique up to precomposition with the
right-multiplication by an element of $S^3$, cf.~\cite[Lemma~3.9]{gego02}.

\subsection{The double covering $S^3\rightarrow ST^*S^2$}
\label{subsection:double}
We think of $S^3$ as the unit sphere in $\HH$, and we identify
$\R^3$ with the space of pure imaginary quaternions.
For $\bfu\in S^3$ and $\bfx\in\R^3$, define $f_{\bfu}(\bfx):=\obfu\bfx\bfu$.
Then $f_{\bfu}$ is in fact an element of $\SO(3)$, and the
map $\bfu\mapsto f_{\bfu}$ is a double
covering $S^3\rightarrow\SO(3)$, see~\cite[Theorem~10.9]{geig16}.

Writing the elements of $\SO(3)$ in matrix form,
we can describe the map $\bfu\mapsto f_{\bfu}$ more explicitly by
\[ \bfu\longmapsto
\begin{pmatrix}
|             & |             & |            \\
\obfu\bfi\bfu & \obfu\bfj\bfu & \obfu\bfk\bfu\\
|             & |             & |            
\end{pmatrix}. \]

\begin{rem}
Notice that the last column is determined by the first two. In fact,
for $\bfx,\bfy\in\R^3\subset\HH$ one has
\[ \bfx\bfy-\bfy\bfx=2\bfx\times\bfy,\]
see~\cite[Lemma~10.8]{geig16}. One then checks easily that
$\obfu\bfi\bfu\times\obfu\bfj\bfu=\obfu\bfk\bfu$.
\end{rem}

Any two columns of a special orthogonal matrix
define an element of $ST^*S^2$, with the model given
in Section~\ref{subsection:structureS2},
and this gives a diffeomorphism
from $\SO(3)$ to $ST^*S^2$. The following choice
of double covering $S^3\rightarrow ST^*S^2$ is made so as to be compatible
with the Hopf fibration and such that the structure forms
on $ST^*S^2$ pull back to the triple of contact forms on $S^3$
induced by the quaternions. We first verify the statement about the
Hopf fibration.

\begin{lem}
The covering
\[ \begin{array}{rccc}
\Phi\co & S^3  & \longrightarrow & ST^*S^2\subset\R^6\\
        & \bfu & \longmapsto     & (\obfu\bfi\bfu,\obfu\bfk\bfu).
\end{array} \]
sends the Hopf fibres
of $S^3$ two-to-one onto the fibres of $ST^*S^2$.
\end{lem}

\begin{proof}
The Hopf fibres are the orbits of the $S^1$-action
$\bfu\mapsto\rme^{\bfi t/2}\bfu$, $t\in\R/4\pi\Z$. Under this action we have
\[ \obfu\bfi\bfu\longmapsto\obfu\rme^{-\bfi t/2}\bfi\rme^{\bfi t/2}\bfu=
\obfu\bfi\bfu,\]
which means that $\Phi$ is fibre-preserving.
On the fibre coordinate of $ST^*S^2$ we have
\[ \obfu\bfk\bfu\longmapsto\obfu\rme^{-\bfi t/2}\bfk\rme^{\bfi t/2}\bfu
=\obfu\bigl(\bfj\sin t+\bfk\cos t\bigr)\bfu.\]
Since $t$ ranges over $\R/4\pi\Z$, each fibre
of $ST^*S^2$ is doubly covered by a Hopf fibre.
Notice that the orientation is positive with respect to the natural
orientation of the fibres of $ST^*S^2$.
\end{proof}

\begin{rem}
We have $f_{\bfv}\circ f_{\bfu}=f_{\bfu\bfv}$. From the
group-theoretic point of view it might be a little more natural
to define the double covering $S^3\rightarrow\SO(3)$ by sending $\bfu$
to the map $\bfx\mapsto\bfu\bfx\obfu$, as was done in~\cite{geig16}.
The preceding lemma explains why, in the present context,
our definition of $f_{\bfu}$ is the preferred one.
\end{rem}

Observe that $\Phi(\bfu)=\Phi(-\bfu)$, so $\Phi$ induces a diffeomorphism
from $\RP^3$, the quotient of $S^3$ under the antipodal map, to
$ST^*S^2$.

\begin{rem}
\label{rem:cartesian}
This lemma implies that $\Phi$ induces a diffeomorphism
on $S^2$ (as the base of the Hopf fibration and of the
unit cotangent bundle, respectively).
We want to check that with respect to an appropriate choice of cartesian
coordinates on $\R^3$ this map is the identity. We have
\begin{eqnarray*}
\obfu\bfi\bfu
 & =  & (u_0^2+u_1^2-u_2^2-u_3^2)\,\bfi+2(u_1u_2-u_0u_3)\,\bfj\\
 &    & \mbox{}+2(u_0u_2+u_1u_3)\,\bfk\\
 & =: & x_3\bfi-x_2\bfj+x_1\bfk. 
\end{eqnarray*}
For $S^2\subset\R^3_{x_1,x_2,x_3}$ consider the stereographic projection
from the south pole $(0,0,-1)$ onto the equatorial $x_1x_2$-plane,
which we identify with~$\C$. Notice that both the 
chosen permutation of cartesian coordinates and this stereographic
projection are orientation-preserving. As image of the
stereographic projection we obtain
\[ \frac{x_1+x_2\bfi}{1+x_3}=
\frac{(u_0u_2+u_1u_3)+(u_0u_3-u_1u_2)\bfi}{u_0^2+u_1^2}=
\frac{u_2+u_3\bfi}{u_0+u_1\bfi},\]
which equals $z_1/z_0$ under the identification of
$\bfu$ with $z_0+z_1\bfj$.
\end{rem}

The following lemma says that the structure forms pull back as desired.

\begin{lem}
\label{lem:pullback}
Let $\lambda_1,\lambda_2$ be the Liouville--Cartan forms on $ST^*S^2$,
and $\alpha$ the connection $1$-form.
Then $\Phi^*\lambda_1=\alpha_{\bfj}$, $\Phi^*\lambda_2=\alpha_{\bfk}$,
and $\Phi^*\alpha=\alpha_{\bfi}$.
\end{lem}

\begin{proof}
The first two equalities are proved by a quaternionic computation that
can be found in Appendix~\ref{appendix:quaternionic}. The connection form
$\alpha$ is determined by the structural equations for $\rmd\lambda_1$
and $\rmd\lambda_2$; likewise, $\alpha_{\bfi}$ is determined
by $\alpha_{\bfj}$ and~$\alpha_{\bfk}$. This yields the third equality.
\end{proof}
\subsection{The quaternionic action}
\label{subsection:q-action}
The unit quaternions $S^3\subset\HH$ form a group, and we consider the
left-action of $S^3$ on itself. For $\bfa\in S^3$ we write
\[ \begin{array}{rccc}
l_{\bfa}\co & S^3  & \longrightarrow & S^3\\
            & \bfu & \longmapsto     & \bfa\bfu.
\end{array}\]

Since we are ultimately concerned with Reeb orbits rather than contact forms,
we formulate the following lemma in terms of the push-forward of forms.

\begin{lem}
The push-forward of the contact form $\alpha_{\bfc}$ under $l_{\bfa}$ is
the contact form $\alpha_{\bfa\bfc\obfa}$. In particular, we
have $Tl_{\bfa}(R_{\bfc})=R_{\bfa\bfc\obfa}$.
\end{lem}

\begin{proof}
The statement about the Reeb vector field is best proved directly. The orbits
of $R_{\bfc}$ are described by the differential equation
$\dot{\gamma}=\bfc\gamma/2$. Then $\bfa\gamma$ satisfies the differential
equation
\[ \frac{\rmd}{\rmd t}(\bfa\gamma)=\bfa\dot{\gamma}=
\frac{1}{2}\bfa\bfc\gamma=\frac{1}{2}(\bfa\bfc\obfa)(\bfa\gamma),\]
so $\bfa\gamma$ is an orbit of $R_{\bfa\bfc\obfa}$.

For the transformation of the contact form, observe that for
$\bfp\in S^3$ and $\bfv\in T_{\bfp}S^3$ we have
$\rmd r_{\bfp}(\bfc\bfv)=\langle\bfp,\bfc\bfv\rangle$, and use
$l_{\bfa}\in\SO(4)$.
\end{proof}

\begin{rem}
\label{rem:right-invariant}
From $\rmd r_{\bfp}(\bfc\bfv)=\langle\bfp,\bfc\bfv\rangle$
one also sees that the forms $\alpha_{\bfc}$ are invariant
under the \emph{right}-action of $S^3$ on itself.
\end{rem}

We now take the Reeb flow of $\alpha_{\bfi}$, which defines
the Hopf fibration, as our point of reference, and we want to understand
how this flow transforms under the $S^3$-action.
Notice that the transformation $\bfi\mapsto\bfa\bfi\obfa$ is
the $\SO(3)$-action described in the preceding section.
The $2$-sphere $S^2\subset\R^3$ of complex structures has an
$S^1$-symmetry fixing~$\bfi$, given by conjugation with
$\bfb=\cos(\varphi/2)+\bfi\sin(\varphi/2)$. The left-action by
$\bfb$ is simply a shift along the Hopf fibres.
Thus, to understand the dynamics of
$R_{\bfa\bfi\obfa}$, notably the projection of
this flow to~$S^2$, it suffices to consider
a circular family of unit quaternions $\bfa$ for which
$\bfa\bfi\obfa$ rotates in the $\bfi\bfj$-plane, say. This
is achieved by
\[ \bfa =\cos\frac{\theta}{2}+\bfk\sin\frac{\theta}{2},\]
in which case
\[ \bfa\bfi\obfa=\bfi\cos\theta+\bfj\sin\theta.\]

\begin{prop}
\label{prop:R-rotate}
For $\bfa=\cos(\theta/2)+\sin(\theta/2)\,\bfk$, $\theta\in]0,\pi[$,
every orbit of $R_{\bfa\bfi\obfa}$ on $S^3$ projects under the Hopf fibration
to a doubly covered circle of latitude of angle~$\theta$ (measured
from the pole defined by the corresponding orbit of~$R_{\bfi}$).
\end{prop}

\begin{proof}
For reasons of symmetry, it suffices to study the orbit of
$R_{\bfa\bfi\obfa}$ obtained via the transformation $l_{\bfa}$ from the
orbit $\gamma$ of $R_{\bfi}$ (that is, the Hopf fibre)
over the point $[1:0]\in\CP^1$.
The orbit $\gamma$ is parametrised by
\[ \gamma(t)=(\rme^{\bfi t},0)=\Bigl(\cos\frac{t}{2},\sin\frac{t}{2},
0,0\Bigr),\;\; t\in\R/4\pi\Z.\]
The transformed orbit is
\[ \bfa\gamma(t)=\Bigl(\cos\frac{\theta}{2}\cos\frac{t}{2},
\cos\frac{\theta}{2}\sin\frac{t}{2},
\sin\frac{\theta}{2}\sin\frac{t}{2},
\sin\frac{\theta}{2}\cos\frac{t}{2}\Bigr).\]
This projects to $\CP^1$ as
\[ \Bigl[\cos\frac{\theta}{2}\Bigl(\cos\frac{t}{2}+\bfi\sin\frac{t}{2}\Bigr):
\bfi\sin\frac{\theta}{2}\Bigl(\cos\frac{t}{2}-\bfi\sin\frac{t}{2}\Bigr)\Bigr]=
\Bigl[1:\bfi\rme^{-\bfi t}\tan\frac{\theta}{2}\Bigr].\]
Figure~\ref{figure:stereographic} shows that this describes
a circle of latitude of angle $\theta$ measured from the south pole
$[1:0]$. The circle is doubly covered, because $t$ ranges over $\R/4\pi\Z$.
\end{proof}

\begin{figure}[h]
\labellist
\small\hair 2pt
\pinlabel $[1:0]$ [tl] at 218 35
\pinlabel $[0:1]$ [bl] at 218 400
\pinlabel $1$ [r] at 216 296
\pinlabel $[1:z]$ [l] at 371 115
\pinlabel $\overline{z}$ [bl] at 312 220
\pinlabel $\theta$ at 228 193
\pinlabel $\theta/2$ [l] at 219 310
\pinlabel $\tan(\theta/2)$ [b] at 406 334
\endlabellist
\centering
\includegraphics[scale=.48]{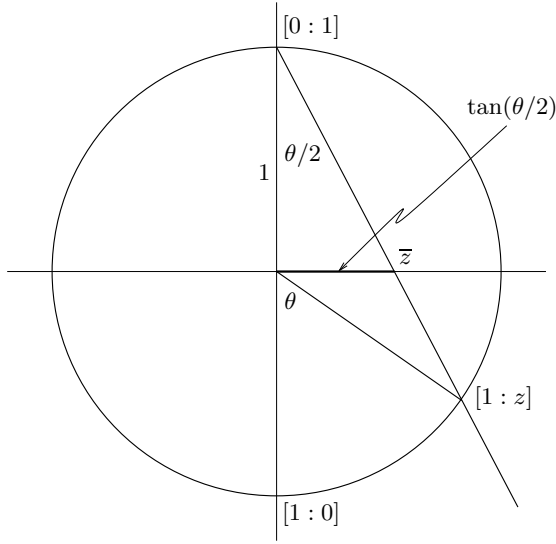}
  \caption{The circle of latitude of angle~$\theta$.}
  \label{figure:stereographic}
\end{figure}

\begin{rem}
For $\theta=\pi$, each Hopf fibre is mapped to the Hopf fibre over the
respective antipodal point in~$S^2$.
\end{rem}
\subsection{Magnetic flows}
\label{subsection:magnetic}
We now want to interpret the flow of the Reeb vector field
$R_{\bfa\bfi\obfa}$ on $S^3$ as the lift of a magnetic flow on
$ST^*S^2$. With $\bfa$ as in Proposition~\ref{prop:R-rotate},
we abbreviate $\alpha_{\bfa\bfi\obfa}$ and $R_{\bfa\bfi\obfa}$
to $\alpha^{\theta}$ and $R^{\theta}$, respectively. With
Lemma~\ref{lem:pullback} we then have
\[ \alpha^{\theta}=\cos\theta\,\alpha_{\bfi}+\sin\theta\,\alpha_{\bfj}=
\Phi^*(\cos\theta\,\alpha+\sin\theta\,\lambda_1).\]
Set
\[ \tilde{\omega}^{\theta}:=
\rmd\alpha^{\theta}=\Phi^*(\cos\theta\,\pi^*\sigma_0+
\sin\theta\,\rmd\lambda_1).\]

\begin{rem}
Notice that $\rmd\lambda_1$ extends naturally as a symplectic form
on the full cotangent bundle as the canonical symplectic form~$\rmd\lambda$.
As discussed in the introduction to this Section~\ref{section:quaternionic},
we may also replace $\pi^*\sigma_0$ by the symplectic $2$-form
$\rmd (\rho\alpha)$ defined on all of $T^*S^2\setminus S^2$. We may
likewise think of $\tilde{\omega}^{\theta}$ as being defined on
$\R^4\setminus\{0\}\cong\R\times S^3$
as the symplectisation of the contact form~$\alpha^{\theta}$.
Thus, we recognise $\tilde{\omega}^{\theta}$ as the double cover of the
symplectic $2$-form used in the definition of magnetic flows.
\end{rem}

By Remark~\ref{rem:Ham-Reeb},
the $S^1$-family of Reeb flows on $S^3$ defined by
the contact forms $\{\alpha^{\theta}\}_{\theta\in\R/2\pi\Z}$
descends to $ST^*S^2=S^3/\Z_2$ to yield an interpolation via
magnetic flows
between the geodesic flow (for a vanishing magnetic term)
and the Hopf flow along the fibres (for an `infinite' magnetic field).
\subsection{Deformation of the magnetic flows}
\label{subsection:deformation}
Recall from Section~\ref{subsection:Hopf}
the description of $\alpha_{\bfi}$ and $R_{\bfi}$
in terms of cartesian coordinates $(x_0,y_0,x_1,y_1)$.
For $\varepsilon\in[0,1)$, we can consider the deformed contact form
\[ \alpha_{\bfi,\varepsilon}:=\frac{2}{1+\varepsilon}
(x_0\,\rmd y_0-y_0\,\rmd x_0)+\frac{2}{1-\varepsilon}
(x_1\,\rmd y_1-y_1\,\rmd x_1) \] 
on $S^3$, with Reeb vector field
\[ R_{\bfi,\varepsilon}=R_{\bfi}+
\frac{\varepsilon}{2}(x_0\partial_{y_0}-y_0\partial_{x_0})-
\frac{\varepsilon}{2}(x_1\partial_{y_1}-y_1\partial_{x_1}).\]
Notice that for irrational values of $\varepsilon$,
with $(1+\varepsilon)/(1-\varepsilon)$ then likewise being irrational, the
only periodic orbits of $R_{\bfi,\varepsilon}$ are
$S^3\cap\{z_1=0\}$ and $S^3\cap\{z_0=0\}$, two circles forming
a Hopf link. In other words, the two periodic orbits of $R_{\bfi}$
that persist after the deformation are the Hopf fibres over
the points $[1:0]$ and $[0:1]$; the speed of the former increases,
the latter is traversed more slowly.

\begin{rem}
The $1$-form $\alpha_{\bfi,\varepsilon}$ on $S^3$ may be
regarded as the contact form induced by
$2\sum_{i=0}^1(x_i\,\rmd y_i-y_i\,\rmd x_i)$ (regarded as a form on $\R^4$)
under the embedding of $S^3$ as an ellipsoid~$E_{\varepsilon}$,
in complex notation,
\[ \begin{array}{ccc}
S^3       & \longrightarrow & \bigl\{ (1+\varepsilon)|z_0|^2+
                              (1-\varepsilon)|z_1|^2=1\bigr\}
                              =:E_{\varepsilon}\\[1.5mm]
(z_0,z_1) & \longmapsto     &  \displaystyle{\Bigl(
                               \frac{z_0}{\sqrt{1+\varepsilon}},
                               \frac{z_1}{\sqrt{1-\varepsilon}}\Bigr)}.
\end{array} \]
\end{rem}

It is useful, as in~\cite{zill82},
to study the flow of $R_{\bfi}$ and the additional summand
in $R_{\bfi,\varepsilon}$ separately. The flow
\[ t\longmapsto \rme^{\bfi t/2}(z_0,z_1)\]
of $R_{\bfi}$ gives us the fibres of the Hopf fibration;
the flow lines of
\[ \frac{1}{2}(x_0\partial_{y_0}-y_0\partial_{x_0})-
\frac{1}{2}(x_1\partial_{y_1}-y_1\partial_{x_1})\] 
are parametrised by
\[ s\longmapsto(\rme^{\bfi s/2}z_0,\rme^{-\bfi s/2}z_1).\]
These two flows commute and define a $T^2$-action on~$S^3$.
The $\R$-action defined by the flow of $R_{\bfi,\varepsilon}$
corresponds to a line in the $2$-torus $T^2=\R^2/(4\pi\Z)^2$ of
slope~$\varepsilon$.

The following geometric observation is not, strictly speaking, relevant
to our discussion, but it is worth noting nonetheless.

\begin{lem}
Let $\beta(s)=(\rme^{\bfi s/2}z_0,\rme^{-\bfi s/2}z_1)$, $s\in\R/4\pi\Z$,
be a flow line of the anti-diagonal $S^1$-action. For any $\bfa\in S^3$,
the rotated flow line $l_{\bfa}(\beta)$ projects under the
Hopf fibration to a circle of latitude on $S^2$ relative to the
poles $[0:1]$ and $[1:0]$.
\end{lem}

\begin{proof}
Write $\bfa=w_0+w_1\bfj$ with $w_0,w_1\in\C$. We compute
\[ (w_0+w_1\bfj)(\rme^{\bfi s/2}z_0+\rme^{-\bfi s/2}z_1\bfj)=
(w_0z_0-w_1\overline{z}_1)\,\rme^{\bfi s/2}+
(w_0z_1+w_1\overline{z}_0)\,\rme^{-\bfi s/2}.\]
This projects to
\[ \Bigl[ 1:\frac{w_0z_1+w_1\overline{z}_0}{w_0z_0-w_1\overline{z}_1}\,
\rme^{-\bfi s}\Bigr],\]
which describes a circle of latitude on $S^2$
as in the proof of Proposition~\ref{prop:R-rotate}; see also
Figure~\ref{figure:stereographic}.
\end{proof}
\subsection{Proof of Theorem~\ref{thm:quaternionic}}
With $\bfa=\cos(\theta/2)+\sin(\theta/2)\,\bfk$ as in
Proposition~\ref{prop:R-rotate} and Section~\ref{subsection:magnetic}
we set $\alpha^{\theta}_{\varepsilon}:=l_{\bfa}(\alpha_{\bfi,\varepsilon})$.
As we have seen, the Reeb flow of $\alpha_{\bfi,\varepsilon}$
has precisely two periodic orbits for any irrational value
of the deformation parameter~$\varepsilon$. Since the
Reeb flow  of $\alpha_{\bfi,\varepsilon}$ is equivariant
with respect to the antipodal map, the induced flow on
$ST^*S^2$ likewise has precisely two periodic orbits, both doubly
covered by the corresponding orbit on~$S^3$. This proves
part (ii) of the theorem, once the Hamiltonian interpretation
in part (i) has been established.

In order to prove part (i), we may work on
$\bigl(\R^4\setminus\{0\},\tilde{\omega}^{\theta}\big)$,
viewed as the symplectisation of $(S^3,\alpha^{\theta})$,
cf.\ Section~\ref{subsection:magnetic}, provided we ensure that
our construction is $\Z_2$-invariant. Since the left-action of
$S^3$ on itself extends to an action on $\R^4\setminus\{0\}$,
it suffices to study the deformation $\alpha_{\bfi,\varepsilon}$
of~$\alpha_{\bfi}$.

On $\R^4\setminus\{0\}=\C^2\setminus\{0\}$ with the symplectic form
\[ \tilde{\omega}^0=4(\rmd x_0\wedge\rmd y_0+\rmd x_1\wedge\rmd y_1)\]
we consider the Hamiltonian functions
\[ H_0(z_0,z_1):=|z_0|^2+|z_1|^2\;\;\;\text{and}\;\;\;
F(z_0,z_1):=|z_0|^2-|z_1|^2,\]
and we set $H_{\varepsilon}:=H_0+\varepsilon F$. The Hamiltonian vector
field of $H_{\varepsilon}$ is
\[ X_{\varepsilon}=\frac{1+\varepsilon}{2}(x_0\partial_{y_0}-
y_0\partial_{x_0})+\frac{1-\varepsilon}{2}(x_1\partial_{y_1}-
y_1\partial_{x_1}).\]
So we recognise the Reeb vector field $R_{\bfi,\varepsilon}$ from 
Section~\ref{subsection:deformation} as the Hamiltonian vector field
of $X_{\varepsilon}$ restricted to the
level set $H_{\varepsilon}^{-1}(1)=E_{\varepsilon}$. Notice that
the form of $X_{\varepsilon}$ does not change under the
map that sends $S^3$ to $E_{\varepsilon}$.

Finally, we want to prove part~(iii). It suffices to
consider $\theta=\pi/2$. Recall that $\alpha^{\pi/2}=\alpha_{\bfj}$.
We may work in the symplectic manifold
\[ \bigl(\R^4\setminus\{0\},\tilde{\omega}^{\pi/2}=
2(\rmd x_1\wedge \rmd x_2+\rmd y_2\wedge\rmd y_1)\bigr).\]
Under the left-action of $\bfa\in S^3$,
the Hamiltonian function $H_{\varepsilon}$ transforms to the function
$H_{\varepsilon}\circ l_{\bfa}^{-1}=H_{\varepsilon}\circ l_{\obfa}$.
The original Hamiltonian function $H_0$ is obviously
invariant, since $l_{\bfa}$ is norm-preserving.
The angle $\theta=\pi/2$ corresponds to the quaternion
$\bfa=\sqrt{2}/2+\bfk\sqrt{2}/2$. We compute
\begin{eqnarray*}
\lefteqn{\Bigl(\frac{\sqrt{2}}{2}-\frac{\sqrt{2}}{2}\bfk\Bigr)\cdot
(x_0+y_0\bfi+x_1\bfj+y_1\bfk)=}\\
 &  & \Bigl(\frac{\sqrt{2}}{2}x_0+\frac{\sqrt{2}}{2}y_1\Bigr)
      +\Bigl(\frac{\sqrt{2}}{2}y_0+\frac{\sqrt{2}}{2}x_1\Bigr)\bfi+\\
 &  &  \Bigl(\frac{\sqrt{2}}{2}x_1-\frac{\sqrt{2}}{2}y_0\Bigr)\bfj
      +\Bigl(\frac{\sqrt{2}}{2}y_1-\frac{\sqrt{2}}{2}x_0\Bigr)\bfk.
\end{eqnarray*}
This yields
\[ F\circ l_{\obfa}(x_0,y_0,x_1,y_1)=x_0y_1+y_0x_1,\]
hence $H_{\varepsilon}$ transforms to
\begin{eqnarray*}
K_{\varepsilon}(x_0,y_0,x_1,y_1)
 & := & H_{\varepsilon}\circ l_{\obfa}(x_0,y_0,x_1,y_1)\\
 & =  & H_0(x_0,y_0,x_1,y_1)+\varepsilon(x_0y_1+y_0x_1).
\end{eqnarray*}
In the notation of the theorem we have $K_{\varepsilon}=
H_{\varepsilon}^{\pi/2}$.

The following homogeneity property will be essential for
the interpretation of the Hamiltonian flow of $K_{\varepsilon}$
as a Finsler geodesic flow.

\begin{lem}
When viewed as a function on $T^*S^2$, the Hamiltonian function
$K_{\varepsilon}$ is homogeneous of degree~$1$ in the fibre
coordinates.
\end{lem}

\begin{proof}
The strict contactomorphism between $(S^3,\alpha_{\bfj})$ and
the double cover of the unit cotangent bundle
$(ST^*S^2,\lambda_1)$ extends to a
symplectomorphism
\[ \tilde{\Phi}\co\bigl(\R^4\setminus\{0\},\tilde{\omega}^{\pi/2}\bigr)
\stackrel{\cong}{\longrightarrow}
\bigl(\widetilde{T^*S^2\setminus S^2},\rmd\lambda\bigr)\]
by sending the flow lines of the Liouville vector field
\[ Y_{\R^4}=\frac{1}{2}(x_0\partial_{x_0}+y_0\partial_{y_0}+
x_1\partial_{x_1}+y_1\partial_{y_1})\]
for $\tilde{\omega}^{\pi/2}$ to those of the fibrewise
radial Liouville vector field $Y_{T^*S^2}$ (see the
discussion before Remark~\ref{rem:Ham-Reeb}), lifted to the
double cover.

A function $G\co T^*S^2\rightarrow\R$ is homogeneous of degree~$1$
in the fibre coordinates precisely if $\rmd G(Y_{T^*S^2})=G$.
Then the pulled-back function $G\circ\tilde{\Phi}$ satisfies
the equation
\[ \rmd(G\circ\tilde{\Phi})(Y_{\R^4})=G\circ\tilde{\Phi}.\]
By Euler's theorem, because of the factor $1/2$ in $Y_{\R^4}$,
this means that $G\circ\tilde{\Phi}$ is homogeneous of degree~$2$
in $(x_0,y_0,x_1,y_1)$. Conversely, a homogeneous function
on $\R^4$ is mapped to a fibrewise homogeneous function on $T^*S^2$
of half the degree of homogeneity.

Since $K_{\varepsilon}$ is homogeneous of degree~$2$ on~$\R^4$,
this proves the lemma.
\end{proof}

\begin{rem}
Since $H_0$ takes the value $1$ on $S^3$, the corresponding
Hamiltonian function on $T^*S^2$ is $\bfv^*\mapsto\|\bfv^*\|_*$.
\end{rem}

As observed in Remark~\ref{rem:geodesic}, the Hamiltonian flow
of $H_0^2/2$ is the geodesic flow. The Hamiltonian flow of
$H_0$ is simply a reparametrisation of it, and on $ST^*S^2$
the two flows coincide, since $X_{H_0^2/2}=H_0^{-1}X_{H_0}$.
Similarly, the Hamiltonian flow of the function $K_{\varepsilon}$
is just a reparametrisation of the one defined by its square.
The latter is fibrewise homogeneous of degree~$2$ and, for $\varepsilon=0$,
fibrewise strictly convex. This condition is preserved for
small values of~$\varepsilon$, in which case $(K_{\varepsilon})^2$
gives rise to a Finsler metric, see~\cite[p.~137]{zill82}
or~\cite{dgz17}.

This completes the proof of Theorem~\ref{thm:quaternionic}.

\begin{rem}
(1) In the Finsler case ($\theta=\pi/2$), the two geodesics that survive the
perturbation equal the great circle on $S^2$ with respect to
the poles $[1:0]$ and $[0:1]$, traversed in opposite directions
(and with different speeds). Since the Finsler metric is non-reversible,
this does indeed count as two geodesics.

(2) If we label the coordinates on $\R^4=\HH$ as
\[ (u_0,u_1,u_2,u_3)\;\;\text{instead of}\;\; (x_0,y_0,x_1,y_1),\]
the perturbation of the Hamiltonian function in the Finsler case
is described by $\varepsilon(u_0u_3+u_1u_2)$.
Computing with the conventions for cartesian
coordinates as in Remark~\ref{rem:cartesian}, one finds
with $\bfx=\obfu\bfi\bfu$ and $\bfy=\obfu\bfk\bfu$ that
\[ -2(u_0u_3+u_1u_2)=x_1y_2-x_2y_1.\]
This means that on $ST^*S^2$ the perturbation term is described by
\[ -\frac{\varepsilon}{2}\langle (y_1,y_2,y_3),(0,-x_3,x_2)\rangle.\]
Since this expression is homogeneous of degree~$1$ in the
fibre coordinates, it describes the perturbation term on the
full cotangent bundle $T^*S^2$.
Thus, up to a constant scale, the perturbation term
is given by the function which sends a point in $T^*S^2$,
i.e.\ a covector, to its evaluation on the vector field
defining the rotation of $S^2$ about the $x_1$-axis.
This is precisely the type of perturbation in Katok's
example as described in~\cite[p.~137]{zill82}.

(3) Observe that the perturbation term, when interpreted dually on
the tangent bundle $TS^2$, is given by a $1$-form $\mu$ on $S^2$ evaluated
on tangent vectors, that is, the Finsler norm of $\bfv\in T_{\bfx}S^2$ is
$\sqrt{g_{0,\bfx}(\bfv,\bfv)}+\mu_{\bfx}(\bfv)$.
Finsler metrics of this type are called Randers metrics.
\end{rem}
\section{Reeb flows with a finite number of periodic orbits}
\label{section:finite}
The example in Section~\ref{subsection:deformation}
of a contact form on $S^3$ with precisely two periodic Reeb
orbits easily generalises to higher dimensions. Given
rationally independent positive real numbers $a_0,\ldots,a_n$,
the $1$-form
\[ \sum_{i=0}^na_i(x_i\,\rmd y_i-y_i\,\rmd x_i)\]
on $\R^{2n+2}$ induces a contact form on the unit sphere
$S^{2n+1}\subset\R^{2n+2}$ whose periodic Reeb orbits are
given by the intersection of that sphere with one
of the $x_iy_i$-planes. So there are precisely $n+1$ periodic
Reeb orbits. Conjecturally, this is the smallest possible
number of periodic Reeb orbits in the given dimension. 

Our aim in this section is to describe contact manifolds
with a finite number of periodic Reeb orbits. In dimension three,
there are some deep results concerning this issue. Taubes~\cite{taub07} has
shown the existence of at least one periodic Reeb orbit on any
closed contact $3$-manifold, thus giving a positive answer
to the Weinstein conjecture in this dimension. This has been
extended by Cristofaro-Gardiner and Hutchings~\cite{cghu16},
who have shown that there will always be at least two
periodic Reeb orbits.

It is an open question whether any contact form on a closed, connected
$3$-manifold that has more than two periodic Reeb orbits
actually possesses infinitely many of them. Under some additional
assumptions, a positive answer to this question has recently been
given by Cristofaro-Gardiner, Hutchings and Pomerleano~\cite{chp17}.
Recall that a contact form is called non-degenerate if the
linearised return map at any periodic Reeb orbit does not have $1$
as an eigenvalue.

\begin{thm}[Cristofaro-Gardiner, Hutchings, Pomerleano]
Let $M$ be a closed, connected $3$-manifold with a non-degenerate
contact form $\alpha$. Assume further that the Euler class
of the contact structure is a torsion element in $H^2(M;\Z)$.
Then $\alpha$ has either two or infinitely many periodic Reeb orbits.
\end{thm}

Combining this with a result of Hutchings and Taubes, it follows that
under the assumptions of non-degeneracy and the Euler class being torsion,
contact forms with precisely two periodic Reeb orbits
exist only on $S^3$ and lens spaces.

Here we show that the situation in higher dimensions is completely
different. There are examples of closed, connected contact manifolds
with high finite numbers of periodic Reeb orbits.

We also consider closed characteristics on hypersurfaces in
symplectic manifolds. Recall that the \emph{characteristics}
of a hypersurface $M$ in a symplectic manifold $(W,\omega)$
are the integral curves of $\ker(\omega|_{TM})$. The characteristics
may also be regarded as the orbits of the Hamiltonian flow defined
by any Hamiltonian function on $W$ having $M$ as a level set.
We write $\omega_{\mathrm{st}}=\sum_{i=1}^{n+1} \rmd x_i\wedge
\rmd y_i$ for the standard symplectic form on~$\R^{2n+2}$.

\begin{thm}
\label{thm:finite}
(a) In any dimension $2n+1\geq 5$, there are closed, connected contact
manifolds with arbitrarily large finite numbers of periodic Reeb orbits.
In dimension five, for instance, any number $\geq 3$ can be so realised.

(b) For any natural number $n\geq 2$ and any non-negative
integer $k$ there is a closed, connected hypersurface in $\R^{2n+2}$ with
precisely $k$ closed characteristics. For $n=1$, any number
$k\geq 2$ can be realised.
\end{thm}

\begin{rem}
Part (b) has also been proved, using more intricate arguments,
by Cieliebak~\cite[Corollary~J]{ciel97}. In fact, Cieliebak
shows that for $k\geq 2$ closed characteristics,
the hypersurface may be taken to be of
so-called confoliation type.
\end{rem}
\subsection{Boothby--Wang bundles}
We begin by considering contact forms on principal
$S^1$-bundles that arise from the Boothby--Wang construction~\cite{bowa58},
cf.~\cite[Section~7.2]{geig08}. Thus, we assume that
$B$ is a closed manifold with a symplectic form $\omega$ such that
the de Rham cohomology class $-[\omega/2\pi]\in H^2_{\mathrm{dR}}(B)$ is
integral, i.e.\ it lies in the image of the homomorphism
$H^2(B;\Z)\rightarrow H^2(B;\R)=H^2_{\mathrm{dR}}(B)$
induced by the inclusion $\Z\rightarrow\R$ of coefficients.
One can then find a connection $1$-form $\alpha$ on the principal $S^1$-bundle
$\pi\co M\rightarrow B$ of Euler class $-[\omega/2\pi]$ with
curvature form~$\omega$, that is, $\rmd\alpha=\pi^*\omega$.
The assumption that $\omega$ be symplectic then translates into
$\alpha$ being a contact form. We write $R$ for the Reeb vector field
of~$\alpha$. The orbits of $R$ are the fibres of the $S^1$-bundle.
Our normalisation of the curvature form corresponds to regarding
$S^1$ as $\R/2\pi\Z$.
\subsection{Lifting Hamiltonian vector fields}
Consider a smooth function $H\co B\rightarrow\R$ and its
Hamiltonian vector field $X=X_{H}$ with respect to
the symplectic form~$\omega$. Our aim is to lift $X$ to a
Hamiltonian vector field $\tX$ with respect to the contact
form $\alpha$ on~$M$. This can be done in such a way that
the corresponding contact Hamiltonian is invariant in the
$R$-direction, so that the flow of $\tX$ preserves the
contact form~$\alpha$, not only the contact structure~$\ker\alpha$.

\begin{lem}
Let $\Xh$ be the horizontal lift of~$X$, that is, the
vector field on $M$ with $\alpha(\Xh)=0$ and $T\pi(\Xh)=X$,
and let $\tilde{H}:=H\circ\pi\co M\rightarrow\R$ be the lift of~$H$.
Then the flow of the vector field $\tX:=\tilde{H}R+\Xh$
preserves~$\alpha$. If $H$ (and hence $\tilde{H}$)
takes positive values only, the lifted vector field $\tX$ is the
Reeb vector field of the rescaled contact form $\alpha/\tilde{H}$.
\end{lem}

\begin{proof}
We compute the Lie derivative, writing $\iota$ for the interior product:
\begin{eqnarray*}
L_{\tX}\alpha & = & \rmd(\alpha(\tX))+\iota_{\tX}\rmd\alpha\\
              & = & \rmd\tilde{H} +\iota_{\Xh}\rmd\alpha\\
              & = & \rmd\tilde{H}+\pi^*(\iota_X\omega)\\
              & = & 0,
\end{eqnarray*}
which proves the invariance of $\alpha$ under the flow of~$\tX$.
The statement about $\tX$ being the Reeb vector field of $\alpha/\tilde{H}$
in the case $\tilde{H}>0$
follows from the theory of contact Hamiltonians \cite[Section~2.3]{geig08},
since the flow of $\tX$ preserves the contact structure
$\ker(\alpha/\tilde{H})$, and $\tX$ evaluates to $1$ on this rescaled contact
form.
\end{proof}

\begin{rem}
Adding a constant $c\in\R$ to $H$ (and hence $\tilde{H}$) does not change
$X$ or its horizontal lift~$\Xh$; the lifted Hamiltonian vector field
$\tX$ changes by~$cR$. In this way we obtain all
possible lifts of $X$ to a vector field whose flow preserves~$\alpha$.
By such a change we can always achieve $\tilde{H}>0$,
since $B$ is assumed to be closed.
\end{rem}
\subsection{Lifting Hamiltonian $S^1$-actions}
We now suppose that the Hamiltonian vector field $X=X_H$ on $(B,\omega)$
induces an action by the circle $S^1=\R/2\pi\Z$.
We normalise $H$ by adding a constant in such a
way that $H>0$ and, at some chosen singularity $p_0$ of~$X$,
the value $H(p_0)$ is a natural number.
The following argument for showing that the lifted vector field $\tX$
likewise induces an $S^1$-action is similar to the one used
in~\cite{lerm02}.

\begin{prop}
A Hamiltonian $S^1$-action on $(B,\omega)$ lifts to a Hamiltonian
$S^1$-action on the total space $(M,\alpha)$ of
the Boothby--Wang bundle.
\end{prop}

\begin{proof}
Over a singularity $p\in B$ of the vector field $X$ defining the
$S^1$-action on $B$, the lifted vector field $\tX$ equals $H(p)R$,
so the $\tX$-orbit through any point in the fibre $\pi^{-1}(p)$
is precisely that fibre, traversed positively. For $p=p_0$,
the fibre is traversed $H(p_0)$ times. In order for the $\tX$-flow
along the fibre over any other singularity $p$ to define an $S^1$-action,
the value $H(p)$ has to be integral. This requirement is
indeed satisfied, as the following argument shows.

Choose a smooth path $\gamma\co [0,1]\rightarrow B$ from $\gamma(0)=p_0$
to $\gamma(1)=p$. By acting on this path with the Hamiltonian
$S^1$-action, we obtain a $2$-sphere $S\subset B$, which we
orient by the ordered frame $\dot{\gamma},X$. Here `sphere'
is understood in the sense of smooth singular theory,
see~\cite[Section~V.5]{bred93}. We then compute
\begin{eqnarray*}
H(p)-H(p_0)
 & = & \int_{\gamma}\rmd H\;\;=\;\;
       -\int_{\gamma}(\iota_X\omega)\\
 & = & \frac{1}{2\pi}\int_{S}\omega\;\;=\;\;
       \left\langle[\omega/2\pi], [S]\right\rangle,
\end{eqnarray*}
which is, up to sign, simply the evaluation of the Euler class on the
homology class~$[S]$, and hence an integer.

For any non-singular point $q\in B$ of $X$ we consider the
$S^1$-orbit $\beta$ through $q$, which is a circle in~$B$, perhaps
multiply covered. We need to show that the $\tX$-path $\tilde{\beta}$
over $\beta$ is likewise a closed loop. By choosing a path $\gamma$
from $p_0$ to $q$ and acting on it as before, we create a disc
$\Delta$ in $B$ with boundary~$\beta$, again in the sense of
smooth singular theory.

The horizontal lift $\beta_{\mathrm{h}}$ of $\beta$ starts at some
point $\tilde{q}$ in the $S^1$-fibre over $q$ and ends at
$\rme^{\rmi h(\beta)}$, where $h(\beta)$ denotes the holonomy of~$\beta$.
This is computed as $h(\beta)=-\int_{\Delta}\omega$, as can be seen
either by applying the theorem of Stokes to a lifted disc
$\tilde{\Delta}$ (with a boundary segment of oriented length $-h(\beta)$
lying in the fibre over~$q$) or by computing explicitly
in a trivialisation of the $S^1$-bundle over~$\Delta$.
Notice that the condition that $-[\omega/2\pi]$ be an integral
cohomology class guarantees that, modulo~$2\pi$, the holonomy
does not depend on the choice of disc $\Delta$ bounded by~$\beta$,
since the evaluation of $\omega$ over any $2$-sphere made up of two such 
discs (one with reversed orientation) lies in $2\pi\Z$.

Now, observing that the positive orientation of $\Delta$ (for the
boundary orientation defined by~$X$) is defined
by the ordered frame $\dot{\gamma},X$, we essentially perform the
previous computation backwards:
\begin{eqnarray*}
h(\beta)\;\;=\;\;-\int_{\Delta}\omega
 & = & 2\pi\int_{\gamma}\iota_X\omega\\
 & = & -2\pi\int_{\gamma}\rmd H\\
 & = & -2\pi\bigl(H(q)-H(p_0)\bigr)\;\;\equiv\;\;
       -2\pi H(q)\;\mbox{\rm mod}\; 2\pi.
\end{eqnarray*}
On the other hand, along $\beta$ the Hamiltonian function takes
the constant value $H(q)$, so the contribution of the $R$-component
of $\tX$ to the flow over a period $2\pi$ is $2\pi H(q)$, which cancels
the holonomy. This shows that $\tilde{\beta}$ is indeed a smooth loop.
\end{proof}

\begin{rem}
There are far more general results about the lifting of group actions
to principal bundles. The lifting of Hamiltonian group actions
to prequantisation line bundles is discussed in \cite{ggk02}, see in
particular Theorem~6.7 and Example~6.10 there, and
\cite[Corollary~1.3]{mund01}.
\end{rem}
\subsection{A torus action on the Boothby--Wang bundle}
We assume that we are in the situation of the preceding section.
Apart from the $S^1$-action on $(M,\alpha)$ lifted from the
Hamiltonian circle action on $(B,\omega)$, we also have the
action given by the Reeb flow along the fibres. We have already seen that
$L_{\tX}\alpha=0$; the condition $L_R\alpha=0$ is immediate from the
defining equations of the Reeb vector field.

\begin{lem}
The flows of $\tX=\tilde{H}R+\Xh$ and $R$ on $M$ commute and hence
define a torus action preserving the contact form~$\alpha$.
\end{lem}

\begin{proof}
We show that $R$ commutes with the horizontal lift
$\Xh$ of~$X$. Since $\tilde{H}$ is constant along the fibres
of the Boothby--Wang bundle, i.e.\ the $R$-orbits, the Reeb vector field
then also commutes with~$\tX$. From
\[ \alpha([R,\Xh])=\alpha(L_R\Xh)=
L_R(\alpha(\Xh))-(L_R\alpha)(\Xh)=0\]
we see that $[R,\Xh]$ is horizontal. But also
\[ T\pi([R,\Xh])=[T\pi(R),T\pi(\Xh)]=0,\]
which means that the Lie bracket $[R,\Xh]$ vanishes.
\end{proof}

Provided the $S^1$-action on $B$ has finitely many fixed points,
we can find a Reeb flow on $M$ with finitely many periodic
orbits by choosing a line of irrational slope in the torus acting on~$M$,
analogous to the argument in Section~\ref{subsection:deformation}.

\begin{prop}
If the Hamiltonian $S^1$-action on $(B,\omega)$ has finitely
many fixed points, a suitably rescaled contact form on the
Boothby--Wang bundle over $(B,\omega)$ has finitely many periodic
Reeb orbits, namely, the fibres over the fixed points.
\end{prop}

\begin{proof}
For $\varepsilon>0$ an irrational number, the flow of
$\tX_{\varepsilon}:=\tX+\varepsilon R$ preserves $\alpha$, and its only
periodic orbits are the fibres over the fixed points. So the desired
contact form is $\alpha/(\tilde{H}+\varepsilon)$.
\end{proof}
\subsection{Hamiltonian $S^1$-actions with finitely many fixed points}
On the complex projective space $\CP^n$, equipped with the
Fubini--Study symplectic form for which $\CP^1\subset\CP^n$
has area~$\pi$, the function
\[ H\bigl([z_0:\ldots:z_n]\bigr)=
\frac{1}{2}\,\frac{w_1|z_1|^2+\cdots+w_n|z_n|^2}{|z_0|^2+\cdots+|z_n|^2},\]
where $w_1,\ldots,w_n$ are integers, is the Hamiltonian for the
$S^1$-action
\[ \rme^{\rmi\theta}[z_0:\ldots:z_n]=
[z_0:\rme^{\rmi w_1\theta}z_1:\ldots:\rme^{\rmi w_n\theta}z_n].\]
When the integers $w_1,\ldots,w_n$ are pairwise distinct and non-zero,
this $S^1$-action has the $n+1$ isolated fixed points
\[ [1:0:\ldots:0],\,\ldots,\,[0:\ldots:0:1].\]
For the action to be effective, we have to assume that the
greatest common divisor of $w_1,\ldots,w_n$ is~$1$.

A Hamiltonian $S^1$-action extends to the blow-up at
any isolated fixed point, see~\cite[Section~IV.3.2]{audi91},
\cite[Section~3]{kake07}. Such a blow-up replaces the one fixed point
by $n$ new fixed points. By iterating this procedure, we can
realise a Hamiltonian $S^1$-action with $n+1+a(n-1)$ fixed points
on the blown-up manifold $\CP^n\#_a\overline{\CP}^n$
for any non-negative integer~$a$.

\begin{rem}
A blow-up of a symplectic manifold $(B,\omega)$ is effected by removing an
open standard ball in a Darboux chart of~$(B,\omega)$, and collapsing the
characteristic line field on the boundary along which the symplectic
form degenerates. This line field defines the Hopf fibration
$S^{2n-1}\rightarrow\CP^{n-1}$, and the collapse produces the
exceptional divisor $\CP^{n-1}$. The size of the chosen ball
(in terms of the symplectic volume form $\omega^n$) determines
the cohomology class of the symplectic form on the blown-up
manifold. By an appropriate choice of the blow-ups we create a symplectic
form $\omega_a$ on $\CP^n\#_a\overline{\CP}^n$ such that
$[\omega_a/2\pi]$ is a rational cohomology class, and by rescaling we
may assume that this class is integral.
\end{rem}
\subsection{Laudenbach's surgery construction}
Laudenbach~\cite{laud97} describes a surgery construction
on hypersurfaces in symplectic
manifolds that allows one to control the Hamiltonian flow
after the surgery. Inside a single Darboux chart, where the
Hamiltonian flow is linear, one performs four simultaneous
embedded surgeries of index~$1$ on the hypersurface.
If the hypersurface has dimension $2n+1$, the belt sphere
of an index~$1$ surgery is $2n$-dimensional; each surgery replaces
two copies of a ball by a cylinder $S^{2n}\times[-1,1]$.
There will be new periodic orbits (i.e.\ closed
characteristics) inside the $(2n-1)$-dimensional
`equatorial' sphere of that belt sphere. In an explicit model
for that surgery, this equatorial sphere can be realised as
an irrational ellipsoid, in which case it contains precisely
$n$ closed characteristics.

There are Hamiltonian orbits that enter and exit the
cylinder $S^{2n}\times[-1,1]$,
and orbits that are trapped inside the cylinder by becoming
asymptotic to the equatorial ellipsoid. If one performs
only a single surgery, one loses control over the orbits that pass
through the cylinder. By performing four such surgeries, arranged
in a clever symmetric fashion, Laudenbach ensures that an orbit inside
the Darboux chart that enters one of the four cylinders
either becomes trapped or does in fact
pass through all four of them and exits as if the surgery cylinders had not
been there. This is a typical feature of plug constructions,
see~\cite{grz16}.

\begin{prop}
On a given $(2n+1)$-dimensional hypersurface in a
symplectic manifold with
isolated closed characteristics, one can perform a Laudenbach surgery
such that the number of closed characteristics on the
new hypersurface has increased by $4n$ or $4n-1$.
\end{prop}

\begin{proof}
If one performs the Laudenbach construction inside a Darboux chart
not traversed by any closed characteristic, one creates a total of $4n$
additional closed characteristics inside the four surgery cylinders.

One may also choose a Darboux chart traversed by exactly one
closed characteristic. In this case one can place the surgery
cylinders in such a way that this characteristic becomes trapped
inside the first cylinder it enters. So we still get $4n$ additional
closed characteristics, but one of the existing ones has been destroyed.
\end{proof}
\subsection{Proof of Theorem~\ref{thm:finite}}
We now put the results of the preceding sections together.
First we find a contact form on a Boothby--Wang bundle over
$\CP^n\#_a\overline{\CP}^n$ with precisely $n+1+a(n-1)$ periodic Reeb
orbits.
In dimension five ($n=2$), this suffices to realise all numbers
$k\geq 3=2+1$ as the number of periodic Reeb orbits on some contact
manifold.
This proves part (a) of the theorem.

In order to prove part~(b), we start from an irrational
ellipsoid in $\R^{2n+2}$ with $n+1$ closed characteristics, and we
modify this number by performing suitable
Laudenbach surgeries on the contact manifold.
This allows us to create a connected hypersurface in
$(\R^{2n+2},\omega_{\mathrm{st}})$ with
\[ n+1+b\cdot(4n-1)+c\cdot 4n\]
closed characteristics for any non-negative integers $b,c$.

By choosing $(b,c)$ in the range
\[ (4n-2,0),\, (4n-3, 1),\,\ldots,\, (1,4n-3),\, (0,4n-2)\]
we can realise $4n-1$ successive integers.
This implies that starting from
\[ n+1+(4n-2)\cdot(4n-1)=16n^2-11n+3\]
we can realise all integers. With the help of
a Hamiltonian plug~\cite{grz16}, one can destroy any isolated
closed characteristic in dimension $2n+1\geq 5$,
so all smaller non-negative integers can likewise
be realised.

In dimension $2n+1=3$, the Hamiltonian
plug described in \cite{grz16} contains two periodic
orbits; in other words, one can always destroy an isolated characteristic
and create two new ones in the process. This gives us any
number $k\geq n+1=2$.

\appendix

\section{Some quaternionic calculations}
\label{appendix:quaternionic}
We write $\langle\,.\,,\,.\,\rangle$ for the standard inner
product on $\R^4=\HH$, and $\Real(\bfa)$ for the real part $a_0$
of a quaternion $\bfa=a_0+a_1\bfi+a_2\bfj+a_3\bfk$.
The conjugate of $\bfa$ is $\overline{\bfa}=a_0-a_1\bfi-a_2\bfj-a_3\bfk$.

We identify $S^3$ with the unit sphere in $\HH$, and $\R^3$ with
the space of pure imaginary quaternions.

In the following calculations we shall freely use a few basic rules
that are easy to verify (or see~\cite[Section~10.4]{geig16}). These are:
\begin{enumerate}
\item[(i)] $\overline{\bfa\cdot\bfb}=\overline{\bfb}\cdot\overline{\bfa}$;
\item[(ii)] $\langle\bfa,\bfb\rangle=\Real(\bfa\cdot\overline{\bfb})$;
\item[(iii)] $\Real(\bfa\cdot\bfb)=
\Real(\bfb\cdot\bfa)$;
\item[(iv)] $\langle\bfx,\bfy\rangle=
\langle\bfu\bfx\obfu,\bfu\bfy\obfu\rangle$
for $\bfx,\bfy\in\R^3$ and $\bfu\in S^3$.
\end{enumerate}
From (ii) it follows that right-multiplication by a unit quaternion defines
an element of $\SO(4)$; with (iii) the same is true for left-multiplication.

\begin{proof}[Proof of Lemma~\ref{lem:pullback}]
We first compute $\alpha_{\bfj}$ in quaternionic notation.
This gives
\begin{eqnarray*}
\bigl(\alpha_{\bfj}\bigr)_{\bfu}(\dot{\bfu})
 & = & -2\bigl(\rmd r\circ\bfj\bigr)_{\bfu}(\dot{\bfu})\\
 & = & -2\langle\bfu,\bfj\dot{\bfu}\rangle\\
 & = & 2\Real(\bfu\dot{\obfu}\bfj)\\
 & = & 2\Real(\bfj\bfu\dot{\obfu}).
\end{eqnarray*}
For $\alpha_{\bfi}$ and $\alpha_{\bfk}$ we have the analogous
expression.

The differential of $\Phi$ is given by
\[ T_{\bfu}\Phi(\dot{\bfu})=(\dot{\obfu}\bfi\bfu+\obfu\bfi\dot{\bfu},
\dot{\obfu}\bfk\bfu+\obfu\bfk\dot{\bfu}).\]
We then compute
\begin{eqnarray*}
\bigl(\Phi^*\lambda_1\bigr)_{\bfu}(\dot{\bfu})
 & = & (\lambda_1)_{\Phi(\bfu)}(\dot{\obfu}\bfi\bfu+\obfu\bfi\dot{\bfu},
       \dot{\obfu}\bfk\bfu+\obfu\bfk\dot{\bfu})\\
 & = & \langle\obfu\bfk\bfu,\dot{\obfu}\bfi\bfu+\obfu\bfi\dot{\bfu}\rangle\\
 & = & \langle\bfk,\bfu\dot{\obfu}\bfi+\bfi\dot{\bfu}\obfu\rangle\\
 & = & \Real(-\bfj\dot{\bfu}\obfu-\bfk\bfu\dot{\obfu}\bfi)\\
 & = & \Real(-\bfj\dot{\bfu}\obfu+\bfj\bfu\dot{\obfu})\\
 & = & 2\Real(\bfj\bfu\dot{\obfu}),
\end{eqnarray*}
where we have used that
the condition $|\bfu|=1$ gives $\dot{\bfu}\obfu+\bfu\dot{\obfu}=0$.
So we have shown that $\Phi^*\lambda_1=\alpha_{\bfj}$.

For the verification of $\Phi^*\lambda_2=\alpha_{\bfk}$
we start from the observation that the
standard complex structure of $S^2$ acts on $T_{\bfx}S^2$ by
$\dot{\bfx}\mapsto \bfx\times\dot{\bfx}$. This gives
\[ \bigl(\Phi^*\lambda_2\bigr)_{\bfu}(\dot{\bfu})=
\langle\obfu\bfi\bfu\times\obfu\bfk\bfu,
\dot{\obfu}\bfi\bfu+\obfu\bfi\dot{\bfu}\rangle=
-\langle\obfu\bfj\bfu,\dot{\obfu}\bfi\bfu+\obfu\bfi\dot{\bfu}\rangle.\]
The further computation is similar.
\end{proof}

\begin{ack}
H.~G.\ would like to take this opportunity to thank Jes\'us Gonzalo
for the long-term collaboration on all aspects of contact circles
and spheres. The influence of that project on the present work is
plainly evident. We also thank Silvia Sabatini for useful conversations
and references about Hamiltonian actions, and Janko Latschev
for spotting an incomplete argument in an earlier version of this paper.
\end{ack}

\end{document}